\documentclass[11pt]{article}

\usepackage{amsmath,amssymb,amsthm,mathrsfs,amsfonts,dsfont,mathabx,txfonts}

\usepackage[dvipsnames]{xcolor}

\usepackage[brazilian]{babel}
\usepackage[utf8]{inputenc}
\usepackage[T1]{fontenc}

\newtheorem{theorem}{Theorem}[section]
\newtheorem{corollary}{Corollary}[section]

\newtheorem{claim}{Claim}[section]

\newcommand{\inlineeqnum}{\refstepcounter{equation}\mbox{~(\theequation)}}

\begin{document}
\thispagestyle{empty}
\begin{center}
{\LARGE Nonlinear maps preserving sums of triple products on $\ast $-algebras}\\
\vspace{.2in}
{\bf Jo\~{a}o Carlos da Motta Ferreira\\ and \\Maria das Gra\c{c}as Bruno Marietto}\\
\vspace{.2in}
{\it Center for Mathematics, Computing and Cognition,\\
Federal University of ABC,\\
Avenida dos Estados, 5001, \\
09210-580, Santo Andr\'e, Brazil.}\\
e-mail: joao.cmferreira@ufabc.edu.br, graca.marietto@ufabc.edu.br
\end{center}

\begin{abstract} Let $\mathcal{A}$ and $\mathcal{B}$ be two unital complex $\ast $-algebras such that $\mathcal{A}$ has a nontrivial projection.  In this paper, we study the structure of bijective nonlinear maps $\Phi :\mathcal{A}\rightarrow \mathcal{B}$ preserving sum of triple products $\alpha _{1} abc+\alpha _{2} a^{*}cb^{*}+\alpha _{3} ba^{*}c +\alpha _{4} cab^{*}+\alpha _{5} bca+\alpha _{6} cb^{*}a^{*},$ where the scalars $\{\alpha _{k}\}_{k=1}^{6}$ are complex numbers satisfying some conditions.
\end{abstract}

\section{Introduction}

Let $\mathcal{A}$ and $\mathcal{B}$ be two complex $\ast $-algebras and $\{\alpha _{k}\}_{k=1}^{6}$ arbitrary complex numbers. We say that a nonlinear map $\Phi :\mathcal{A}\rightarrow \mathcal{B}$ {\it preserves sum of triple products $\alpha _{1} abc+\alpha _{2} a^{*}cb^{*}+\alpha _{3} ba^{*}c +\alpha _{4} cab^{*}+\alpha _{5} bca+\alpha _{6} cb^{*}a^{*}$} if
{\allowdisplaybreaks\begin{align}\allowdisplaybreaks\label{fundident}
&\Phi (\alpha _{1} abc+\alpha _{2} a^{*}cb^{*}+\alpha _{3} ba^{*}c +\alpha _{4} cab^{*}+\alpha _{5} bca+\alpha _{6} cb^{*}a^{*})\nonumber \\
&=\alpha _{1}\Phi (a)\Phi (b)\Phi (c)+\alpha _{2}\Phi (a)^{*}\Phi (c)\Phi (b)^{*}+\alpha _{3}\Phi (b)\Phi (a)^{*}\Phi (c)\nonumber \\
&+\alpha _{4}\Phi (c)\Phi (a)\Phi (b)^{*}+\alpha _{5}\Phi (b)\Phi (c)\Phi (a)+\alpha _{6}\Phi (c)\Phi (b)^{*}\Phi (a)^{*},
\end{align}}
for all elements $a,b,c\in \mathcal{A}.$

These kinds of maps are related to nonlinear maps preserving the Jordan (resp. Lie, mixed) triple $\ast $-product which have been studied by many authors (for example, see the works \cite{Li1}, \cite{Li2}, \cite{Li3}, \cite{Li4}, \cite{Zhao} and the references therein). In particular, Li, Chen and Wang \cite{Li1} and Zhao and Li \cite{Zhao} studied the structure of the bijective nonlinear maps preserving Jordan triple $\ast $-products on factor von Neumann algebras. Note that these maps satisfy (\ref{fundident}), for convenient scalars $\alpha _{k}$ $(k=1,2, \cdots ,6).$ Motivated by these results, in this paper we will study the structure of bijective nonlinear maps $\Phi ,$ from a unital prime $\ast $-algebra $\mathcal{A}$ having a nontrivial projection to a unital $\ast $-algebra $\mathcal{B},$ preserving sum of triple products $\alpha _{1} abc+\alpha _{2} a^{*}cb^{*}+\alpha _{3} ba^{*}c +\alpha _{4} cab^{*}+\alpha _{5} bca+\alpha _{6} cb^{*}a^{*},$ where $\{\alpha _{k}\}_{k=1}^{6}$ are complex numbers satisfying certain conditions. At the end of this paper, we make a contribution to the problem of structure characterization of the nonlinear maps preserving triple $\ast $-products, on unital $\ast $-algebras, as originated from the works \cite{Li1} and \cite{Zhao}.  

Our main result reads as follows.

\begin{theorem}\label{thm11} Let $\{\alpha _{k}\}_{k=1}^{6}$ be complex numbers satisfying the condition $|\alpha _{1}+\alpha _{3}+\alpha _{5}|-|\alpha _{2}+\alpha _{4}+\alpha _{6}|\neq 0,$ $\mathcal{A}$ and $\mathcal{B}$ two unital complex $\ast $-algebras with $1_{\mathcal{A}}$ and $1_{\mathcal{B}}$ their multiplicative identities, respectively, and such that $\mathcal{A}$ is prime and has a nontrivial projection. Then every bijective nonlinear map $\Phi :\mathcal{A}\rightarrow \mathcal{B}$ preserving sum of triple products $\alpha _{1} abc+\alpha _{2} a^{*}cb^{*}+\alpha _{3} ba^{*}c +\alpha _{4} cab^{*}+\alpha _{5} bca+\alpha _{6} cb^{*}a^{*}$ is additive. In addition, if (i) $\Phi (1_{\mathcal{A}})$ is a projection of $\mathcal{B}$ and (ii) $\alpha _{2}+\alpha _{4}+\alpha _{6}\neq 0$ and $\Phi ((\alpha _{2}+\alpha _{4}+\alpha _{6})a)=(\alpha _{2}+\alpha _{4}+\alpha _{6})\Phi (a),$ for all element $a\in \mathcal{A},$ then $\Phi $ is a $\ast $-ring isomorphism.
\end{theorem}

\section{The proof of main result}

In order to prove the Theorem \ref{thm11} we need to prove several Claims. We begin with a Claim, whose proof is easy and is omitted.

\begin{claim}\label{s21} $\Phi (0)=0.$
\end{claim}

The following well known result will be used throughout this paper: Let $p_{1}$ be any nontrivial projection of $\mathcal{A}$ and write $p_{2}=1_{\mathcal{A}}-p_{1}.$ Then $\mathcal{A}$ has a Peirce decomposition $\mathcal{A}=\mathcal{A}_{11}\oplus \mathcal{A}_{12}\oplus \mathcal{A}_{21}\oplus \mathcal{A}_{22},$ where $\mathcal{A}_{ij}=p_{i}\mathcal{A}p_{j}$ $(i,j=1,2) ,$ satisfying the following multiplicative relations: $\mathcal{A}_{ij}\mathcal{A}_{kl}\subseteq \delta _{jk} \mathcal{A}_{il},$ where $\delta _{jk}$ is the {\it Kronecker delta function}.

\begin{claim}\label{s22} For every $a_{ii}\in \mathcal{A}_{ii},$ $b_{ij}\in \mathcal{A}_{ij}$ and $c_{ji}\in \mathcal{A}_{ji}$ $(i\neq j;i,j=1,2)$ we have: (i) $\Phi (a_{ii}+b_{ij})=\Phi (a_{ii})+\Phi (b_{ij})$ and (ii) $\Phi (a_{ii}+c_{ji})=\Phi (a_{ii})+\Phi (c_{ji}).$
\end{claim}
\begin{proof} Let $u=u_{ii}+u_{ij}+u_{ji}+u_{jj}=\Phi ^{-1}(\Phi (a_{ii}+b_{ij})-\Phi (a_{ii})-\Phi (b_{ij}))\in \mathcal{A}$ $(i\neq j;i,j=1,2).$ According to the definition of $\Phi ,$ we have
{\allowdisplaybreaks\begin{align*}
&\Phi (\alpha _{1} 1_{\mathcal{A}}up_{j}+\alpha _{2} 1_{\mathcal{A}}^{*}p_{j}u^{*}+\alpha _{3} u1_{\mathcal{A}}^{*}p_{j}+\alpha _{4} p_{j}1_{\mathcal{A}}u^{*}+\alpha _{5} up_{j}1_{\mathcal{A}}+\alpha _{6} p_{j}u^{*}1_{\mathcal{A}}^{*})\\
&=\alpha _{1} \Phi (1_{\mathcal{A}})\Phi (u)\Phi (p_{j})+\alpha _{2} \Phi (1_{\mathcal{A}})^{*}\Phi (p_{j})\Phi (u)^{*}+\alpha _{3} \Phi (u)\Phi (1_{\mathcal{A}})^{*}\Phi (p_{j})\\
&+\alpha _{4} \Phi (p_{j})\Phi (1_{\mathcal{A}})\Phi (u)^{*}+\alpha _{5} \Phi (u)\Phi (p_{j})\Phi (1_{\mathcal{A}})+\alpha _{6} \Phi (p_{j})\Phi (u)^{*}\Phi (1_{\mathcal{A}})^{*}\\
&=\alpha _{1} \Phi (1_{\mathcal{A}})\Phi (a_{ii}+b_{ij})\Phi (p_{j})+\alpha _{2} \Phi (1_{\mathcal{A}})^{*}\Phi (p_{j})\Phi (a_{ii}+b_{ij})^{*}\\
&+\alpha _{3} \Phi (a_{ii}+b_{ij})\Phi (1_{\mathcal{A}})^{*}\Phi (p_{j})+\alpha _{4} \Phi (p_{j})\Phi (1_{\mathcal{A}})\Phi (a_{ii}+b_{ij})^{*}\\
&+\alpha _{5} \Phi (a_{ii}+b_{ij})\Phi (p_{j})\Phi (1_{\mathcal{A}})+\alpha _{6} \Phi (p_{j})\Phi (a_{ii}+b_{ij})^{*}\Phi (1_{\mathcal{A}})^{*}\\
&-\alpha _{1} \Phi (1_{\mathcal{A}})\Phi (a_{ii})\Phi (p_{j})-\alpha _{2} \Phi (1_{\mathcal{A}})^{*}\Phi (p_{j})\Phi (a_{ii})^{*}-\alpha _{3} \Phi (a_{ii})\Phi (1_{\mathcal{A}})^{*}\Phi (p_{j})\\
&-\alpha _{4} \Phi (p_{j})\Phi (1_{\mathcal{A}})\Phi (a_{ii})^{*}-\alpha _{5} \Phi (a_{ii})\Phi (p_{j})\Phi (1_{\mathcal{A}})-\alpha _{6} \Phi (p_{j})\Phi (a_{ii})^{*}\Phi (1_{\mathcal{A}})^{*}\\
&-\alpha _{1} \Phi (1_{\mathcal{A}})\Phi (b_{ij})\Phi (p_{j})-\alpha _{2} \Phi (1_{\mathcal{A}})^{*}\Phi (p_{j})\Phi (b_{ij})^{*}-\alpha _{3} \Phi (b_{ij})\Phi (1_{\mathcal{A}})^{*}\Phi (p_{j})\\
&-\alpha _{4} \Phi (p_{j})\Phi (1_{\mathcal{A}})\Phi (b_{ij})^{*}-\alpha _{5} \Phi (b_{ij})\Phi (p_{j})\Phi (1_{\mathcal{A}})-\alpha _{6} \Phi (p_{j})\Phi (b_{ij})^{*}\Phi (1_{\mathcal{A}})^{*}\\
&=\Phi (\alpha _{1} 1_{\mathcal{A}}(a_{ii}+b_{ij})p_{j}+\alpha _{2} 1_{\mathcal{A}}^{*}p_{j}(a_{ii}+b_{ij})^{*}+\alpha _{3} (a_{ii}+b_{ij})1_{\mathcal{A}}^{*}p_{j}\\
&+\alpha _{4} p_{j}1_{\mathcal{A}}(a_{ii}+b_{ij})^{*}+\alpha _{5} (a_{ii}+b_{ij})p_{j}1_{\mathcal{A}}+\alpha _{6} p_{j}(a_{ii}+b_{ij})^{*}1_{\mathcal{A}}^{*})\\
&-\Phi (\alpha _{1} 1_{\mathcal{A}}a_{ii}p_{j}+\alpha _{2} 1_{\mathcal{A}}^{*}p_{j}a_{ii}^{*}+\alpha _{3} a_{ii}1_{\mathcal{A}}^{*}p_{j}+\alpha _{4} p_{j}1_{\mathcal{A}}a_{ii}^{*}+\alpha _{5} a_{ii}p_{j}1_{\mathcal{A}}\\
&+\alpha _{6} p_{j}a_{ii}^{*}1_{\mathcal{A}}^{*})-\Phi (\alpha _{1} 1_{\mathcal{A}}b_{ij}p_{j}+\alpha _{2} 1_{\mathcal{A}}^{*}p_{j}b_{ij}^{*}+\alpha _{3} b_{ij}1_{\mathcal{A}}^{*}p_{j}+\alpha _{4} p_{j}1_{\mathcal{A}}b_{ij}^{*}\\
&+\alpha _{5} b_{ij}p_{j}1_{\mathcal{A}}+\alpha _{6} p_{j}b_{ij}^{*}1_{\mathcal{A}}^{*})=0.
\end{align*}}
Since $\Phi $ is injective we deduce that $\alpha _{1} 1_{\mathcal{A}}up_{j}+\alpha _{2} 1_{\mathcal{A}}^{*}p_{j}u^{*}+\alpha _{3} u1_{\mathcal{A}}^{*}p_{j}+\alpha _{4} p_{j}1_{\mathcal{A}}u^{*}+\alpha _{5} up_{j}1_{\mathcal{A}}+\alpha _{6} p_{j}u^{*}1_{\mathcal{A}}^{*}=0.$ It follows that $(\alpha _{1} +\alpha _{3} +\alpha _{5})u_{ij}+(\alpha _{2} +\alpha _{4} +\alpha _{6})u_{ij}^{*}+(\alpha _{1} +\alpha _{3} +\alpha _{5})u_{jj}+(\alpha _{2} +\alpha _{4} +\alpha _{6})u_{jj}^{*}=0 \inlineeqnum\label{eqn:idt2}.$ Next, applying the involution $\ast $ to the identity (\ref{eqn:idt2}) we get $(\overline{\alpha _{2} +\alpha _{4} +\alpha _{6}})u_{ij}+(\overline{\alpha _{1} +\alpha _{3} +\alpha _{5}})u_{ij}^{*}+(\overline{\alpha _{2} +\alpha _{4} +\alpha _{6}})u_{jj}+(\overline{\alpha _{1} +\alpha _{3} +\alpha _{5}})u_{jj}^{*}=0 \inlineeqnum\label{eqn:idt3}.$ Also, multiplying (\ref{eqn:idt2}) by the scalar $(\overline{\alpha _{1}+\alpha _{3}+\alpha _{5}}),$ (\ref{eqn:idt3}) by the scalar $(\alpha _{2}+\alpha _{4}+\alpha _{6})$ and subtracting the resulting identities, we arrive at $(|\alpha _{1}+\alpha _{3}+\alpha _{5}|^{2}-|\alpha _{2} +\alpha _{4} +\alpha _{6}|^ {2})(u_{ij}+u_{jj})=0.$ This shows that $u_{ij}+u_{jj}=0$ which results that $u_{ij}=0$ and $u_{jj}=0,$ by directness of the Peirce decomposition. Now, we have 
{\allowdisplaybreaks\begin{align*}
&\Phi (\alpha _{1} 1_{\mathcal{A}}up_{i}+\alpha _{2} 1_{\mathcal{A}}^{*}p_{i}u^{*}+\alpha _{3} u1_{\mathcal{A}}^{*}p_{i}+\alpha _{4} p_{i}1_{\mathcal{A}}u^{*}+\alpha _{5} up_{i}1_{\mathcal{A}}+\alpha _{6} p_{i}u^{*}1_{\mathcal{A}}^{*})\\
&=\alpha _{1} \Phi (1_{\mathcal{A}})\Phi (u)\Phi (p_{i})+\alpha _{2} \Phi (1_{\mathcal{A}})^{*}\Phi (p_{i})\Phi (u)^{*}+\alpha _{3} \Phi (u)\Phi (1_{\mathcal{A}})^{*}\Phi (p_{i})\\
&+\alpha _{4} \Phi (p_{i})\Phi (1_{\mathcal{A}})\Phi (u)^{*}+\alpha _{5} \Phi (u)\Phi (p_{i})\Phi (1_{\mathcal{A}})+\alpha _{6} \Phi (p_{i})\Phi (u)^{*}\Phi (1_{\mathcal{A}})^{*}\\
&=\alpha _{1} \Phi (1_{\mathcal{A}})\Phi (a_{ii}+b_{ij})\Phi (p_{i})+\alpha _{2} \Phi (1_{\mathcal{A}})^{*}\Phi (p_{i})\Phi (a_{ii}+b_{ij})^{*}\\
&+\alpha _{3} \Phi (a_{ii}+b_{ij})\Phi (1_{\mathcal{A}})^{*}\Phi (p_{i})+\alpha _{4} \Phi (p_{i})\Phi (1_{\mathcal{A}})\Phi (a_{ii}+b_{ij})^{*}\\
&+\alpha _{5} \Phi (a_{ii}+b_{ij})\Phi (p_{i})\Phi (1_{\mathcal{A}})+\alpha _{6} \Phi (p_{i})\Phi (a_{ii}+b_{ij})^{*}\Phi (1_{\mathcal{A}})^{*}\\
&-\alpha _{1} \Phi (1_{\mathcal{A}})\Phi (a_{ii})\Phi (p_{i})-\alpha _{2} \Phi (1_{\mathcal{A}})^{*}\Phi (p_{i})\Phi (a_{ii})^{*}-\alpha _{3} \Phi (a_{ii})\Phi (1_{\mathcal{A}})^{*}\Phi (p_{i})\\
&-\alpha _{4} \Phi (p_{i})\Phi (1_{\mathcal{A}})\Phi (a_{ii})^{*}-\alpha _{5} \Phi (a_{ii})\Phi (p_{i})\Phi (1_{\mathcal{A}})-\alpha _{6} \Phi (p_{i})\Phi (a_{ii})^{*}\Phi (1_{\mathcal{A}})^{*}\\
&-\alpha _{1} \Phi (1_{\mathcal{A}})\Phi (b_{ij})\Phi (p_{i})-\alpha _{2} \Phi (1_{\mathcal{A}})^{*}\Phi (p_{i})\Phi (b_{ij})^{*}-\alpha _{3} \Phi (b_{ij})\Phi (1_{\mathcal{A}})^{*}\Phi (p_{i})\\
&-\alpha _{4} \Phi (p_{i})\Phi (1_{\mathcal{A}})\Phi (b_{ij})^{*}-\alpha _{5} \Phi (b_{ij})\Phi (p_{i})\Phi (1_{\mathcal{A}})-\alpha _{6} \Phi (p_{i})\Phi (b_{ij})^{*}\Phi (1_{\mathcal{A}})^{*}\\
&=\Phi (\alpha _{1} 1_{\mathcal{A}}(a_{ii}+b_{ij})p_{i}+\alpha _{2} 1_{\mathcal{A}}^{*}p_{i}(a_{ii}+b_{ij})^{*}+\alpha _{3} (a_{ii}+b_{ij})1_{\mathcal{A}}^{*}p_{i}\\
&+\alpha _{4} p_{i}1_{\mathcal{A}}(a_{ii}+b_{ij})^{*}+\alpha _{5} (a_{ii}+b_{ij})p_{i}1_{\mathcal{A}}+\alpha _{6} p_{i}(a_{ii}+b_{ij})^{*}1_{\mathcal{A}}^{*})\\
&-\Phi (\alpha _{1} 1_{\mathcal{A}}a_{ii}p_{i}+\alpha _{2} 1_{\mathcal{A}}^{*}p_{i}a_{ii}^{*}+\alpha _{3} a_{ii}1_{\mathcal{A}}^{*}p_{i}+\alpha _{4} p_{i}1_{\mathcal{A}}a_{ii}^{*}+\alpha _{5} a_{ii}p_{i}1_{\mathcal{A}}\\
&+\alpha _{6} p_{i}a_{ii}^{*}1_{\mathcal{A}}^{*})-\Phi (\alpha _{1} 1_{\mathcal{A}}b_{ij}p_{i}+\alpha _{2} 1_{\mathcal{A}}^{*}p_{i}b_{ij}^{*}+\alpha _{3} b_{ij}1_{\mathcal{A}}^{*}p_{i}+\alpha _{4} p_{i}1_{\mathcal{A}}b_{ij}^{*}\\
&+\alpha _{5} b_{ij}p_{i}1_{\mathcal{A}}+\alpha _{6} p_{i}b_{ij}^{*}1_{\mathcal{A}}^{*})=0
\end{align*}}
which leads directly to identity $\alpha _{1} 1_{\mathcal{A}}up_{i}+\alpha _{2} 1_{\mathcal{A}}^{*}p_{i}u^{*}+\alpha _{3} u1_{\mathcal{A}}^{*}p_{i}+\alpha _{4} p_{i}1_{\mathcal{A}}u^{*}+\alpha _{5} up_{i}1_{\mathcal{A}}+\alpha _{6} p_{i}u^{*}1_{\mathcal{A}}^{*}=0.$ This imples that $(\alpha _{1}+\alpha _{3}+\alpha _{5})u_{ii}+(\alpha _{2}+\alpha _{4}+\alpha _{6})u_{ii}^{*}+(\alpha _{1}+\alpha _{3}+\alpha _{5})u_{ji}+(\alpha _{2}+\alpha _{4}+\alpha _{6})u_{ji}^{*}=0 \inlineeqnum\label{eqn:idt4}.$ Next, applying the involution $\ast $ to the identity (\ref{eqn:idt4}) we get $(\overline{\alpha _{2}+\alpha _{4}+\alpha _{6}})u_{ii}+(\overline{\alpha _{1}+\alpha _{3}+\alpha _{5}})u_{ii}^{*}+(\overline{\alpha _{2}+\alpha _{4}+\alpha _{6}})u_{ji}+(\overline{\alpha _{1}+\alpha _{3}+\alpha _{5}})u_{ji}^{*}=0 \inlineeqnum\label{eqn:idt5}.$ Also, multiplying (\ref{eqn:idt4}) by the scalar $(\overline{\alpha _{1}+\alpha _{3}+\alpha _{5}}),$ (\ref{eqn:idt5}) by the scalar $(\alpha _{2}+\alpha _{4}+\alpha _{6})$ and subtracting the resulting identities, we obtain $(|\alpha _{1}+\alpha _{3}+\alpha _{5}|^{2}-|\alpha _{2}+\alpha _{4}+\alpha _{6}|^ {2})(u_{ii}+u_{ji})=0.$ This results that $u_{ii}+u_{ji}=0$ which shows that $u_{ii}=0$ and $u_{ji}=0.$ As a consequence, we conclude that $u=0.$  Therefore the Claim is proved.

The proof of (ii) follows from similar arguments.
\end{proof}

\begin{claim}\label{s23}  For every $a_{ii}\in \mathcal{A}_{ii}$, $b_{ij}\in \mathcal{A}_{ij}$, $c_{ji}\in \mathcal{A}_{ji}$ and  $d_{jj}\in \mathcal{A}_{jj}$ $(i\neq j;i,j=1,2)$ we have: $\Phi (a_{ii}+b_{ij}+c_{ji}+d_{jj})=\Phi (a_{ii})+\Phi (b_{ij})+\Phi (c_{ji})+\Phi (d_{jj}).$
\end{claim}
\begin{proof} Let $u=u_{ii}+u_{ij}+u_{ji}+u_{jj}=\Phi ^{-1}(\Phi (a_{ii}+b_{ij}+c_{ji}+d_{jj})-\Phi (a_{ii})-\Phi (b_{ij})-\Phi (c_{ji})-\Phi (d_{jj}))=\Phi ^{-1}(\Phi (a_{ii}+b_{ij}+c_{ji}+d_{jj})-\Phi (a_{ii}+c_{ji})-\Phi (b_{ij}+d_{jj}))\in \mathcal{A},$ by Claim \ref{s22}. Then
{\allowdisplaybreaks\begin{align*}
&\Phi (\alpha _{1} 1_{\mathcal{A}}up_{j}+\alpha _{2} 1_{\mathcal{A}}^{*}p_{j}u^{*}+\alpha _{3} u1_{\mathcal{A}}^{*}p_{j}+\alpha _{4} p_{j}1_{\mathcal{A}}u^{*}+\alpha _{5} up_{j}1_{\mathcal{A}}+\alpha _{6} p_{j}u^{*}1_{\mathcal{A}}^{*})\\
&=\alpha _{1} \Phi (1_{\mathcal{A}})\Phi (u)\Phi (p_{j})+\alpha _{2} \Phi (1_{\mathcal{A}})^{*}\Phi (p_{j})\Phi (u)^{*}+\alpha _{3} \Phi (u)\Phi (1_{\mathcal{A}})^{*}\Phi (p_{j})\\
&+\alpha _{4} \Phi (p_{j})\Phi (1_{\mathcal{A}})\Phi (u)^{*}+\alpha _{5} \Phi (u)\Phi (p_{j})\Phi (1_{\mathcal{A}})+\alpha _{6} \Phi (p_{j})\Phi (u)^{*}\Phi (1_{\mathcal{A}})^{*}\\
&=\alpha _{1} \Phi (1_{\mathcal{A}})\Phi (a_{ii}+b_{ij}+c_{ji}+d_{jj})\Phi (p_{j})\\
&+\alpha _{2} \Phi (1_{\mathcal{A}})^{*}\Phi (p_{j})\Phi (a_{ii}+b_{ij}+c_{ji}+d_{jj})^{*}\\
&+\alpha _{3} \Phi (a_{ii}+b_{ij}+c_{ji}+d_{jj})\Phi (1_{\mathcal{A}})^{*}\Phi (p_{j})\\
&+\alpha _{4} \Phi (p_{j})\Phi (1_{\mathcal{A}})\Phi (a_{ii}+b_{ij}+c_{ji}+d_{jj})^{*}\\
&+\alpha _{5} \Phi (a_{ii}+b_{ij}+c_{ji}+d_{jj})\Phi (p_{j})\Phi (1_{\mathcal{A}})\\
&+\alpha _{6} \Phi (p_{j})\Phi (a_{ii}+b_{ij}+c_{ji}+d_{jj})^{*}\Phi (1_{\mathcal{A}})^{*}\\
&-\alpha _{1} \Phi (1_{\mathcal{A}})\Phi (a_{ii}+c_{ji})\Phi (p_{j})-\alpha _{2} \Phi (1_{\mathcal{A}})^{*}\Phi (p_{j})\Phi (a_{ii}+c_{ji})^{*}\\
&-\alpha _{3} \Phi (a_{ii}+c_{ji})\Phi (1_{\mathcal{A}})^{*}\Phi (p_{j})-\alpha _{4} \Phi (p_{j})\Phi (1_{\mathcal{A}})\Phi (a_{ii}+c_{ji})^{*}\\
&-\alpha _{5} \Phi (a_{ii}+c_{ji})\Phi (p_{j})\Phi (1_{\mathcal{A}})-\alpha _{6} \Phi (p_{j})\Phi (a_{ii}+c_{ji})^{*}\Phi (1_{\mathcal{A}})^{*}\\
&-\alpha _{1} \Phi (1_{\mathcal{A}})\Phi (b_{ij}+d_{jj})\Phi (p_{j})-\alpha _{2} \Phi (1_{\mathcal{A}})^{*}\Phi (p_{j})\Phi (b_{ij}+d_{jj})^{*}\\
&-\alpha _{3} \Phi (b_{ij}+d_{jj})\Phi (1_{\mathcal{A}})^{*}\Phi (p_{j})-\alpha _{4} \Phi (p_{j})\Phi (1_{\mathcal{A}})\Phi (b_{ij}+d_{jj})^{*}\\
&-\alpha _{5} \Phi (b_{ij}+d_{jj})\Phi (p_{j})\Phi (1_{\mathcal{A}})-\alpha _{6} \Phi (p_{j})\Phi (b_{ij}+d_{jj})^{*}\Phi (1_{\mathcal{A}})^{*}\\
&=\Phi (\alpha _{1} 1_{\mathcal{A}}(a_{ii}+b_{ij}+c_{ji}+d_{jj})p_{j}+\alpha _{2} 1_{\mathcal{A}}^{*}p_{j}(a_{ii}+b_{ij}+c_{ji}+d_{jj})^{*}\\
&+\alpha _{3} (a_{ii}+b_{ij}+c_{ji}+d_{jj})1_{\mathcal{A}}^{*}p_{j}+\alpha _{4} p_{j}1_{\mathcal{A}}(a_{ii}+b_{ij}+c_{ji}+d_{jj})^{*}\\
&+\alpha _{5} (a_{ii}+b_{ij}+c_{ji}+d_{jj})p_{j}1_{\mathcal{A}}+\alpha _{6} p_{j}(a_{ii}+b_{ij}+c_{ji}+d_{jj})^{*}1_{\mathcal{A}}^{*})\\
&-\Phi (\alpha _{1} 1_{\mathcal{A}}(a_{ii}+c_{ji})p_{j}+\alpha _{2} 1_{\mathcal{A}}^{*}p_{j}(a_{ii}+c_{ji})^{*}+\alpha _{3} (a_{ii}+c_{ji})1_{\mathcal{A}}^{*}p_{j}\\
&+\alpha _{4} p_{j}1_{\mathcal{A}}(a_{ii}+c_{ji})^{*}+\alpha _{5} (a_{ii}+c_{ji})p_{j}1_{\mathcal{A}}+\alpha _{6} p_{j}(a_{ii}+c_{ji})^{*}1_{\mathcal{A}}^{*})\\
&-\Phi (\alpha _{1} 1_{\mathcal{A}}(b_{ij}+d_{jj})p_{j}+\alpha _{2} 1_{\mathcal{A}}^{*}p_{j}(b_{ij}+d_{jj})^{*}+\alpha _{3} (b_{ij}+d_{jj})1_{\mathcal{A}}^{*}p_{j}\\
&+\alpha _{4} p_{j}1_{\mathcal{A}}(b_{ij}+d_{jj})^{*}+\alpha _{5} (b_{ij}+d_{jj})p_{j}1_{\mathcal{A}}+\alpha _{6} p_{j}(b_{ij}+d_{jj})^{*}1_{\mathcal{A}}^{*})=0.
\end{align*}}
This implies that $\alpha _{1} 1_{\mathcal{A}}up_{j}+\alpha _{2} 1_{\mathcal{A}}^{*}p_{j}u^{*}+\alpha _{3} u1_{\mathcal{A}}^{*}p_{j}+\alpha _{4} p_{j}1_{\mathcal{A}}u^{*}+\alpha _{5} up_{j}1_{\mathcal{A}}+\alpha _{6} p_{j}u^{*}1_{\mathcal{A}}^{*}=0.$ It follows that $(\alpha _{1} +\alpha _{3} +\alpha _{5})(u_{ij}+u_{jj})+(\alpha _{2} +\alpha _{4} +\alpha _{6})(u_{ij}+u_{jj})^{*}=0 \inlineeqnum\label{eqn:idt6}.$ Also, by application the involution $\ast $ on (\ref{eqn:idt6}) we obtain the identity $(\overline{\alpha _{2} +\alpha _{4} +\alpha _{6}})(u_{ij}+u_{jj})+(\overline{\alpha _{1} +\alpha _{3} +\alpha _{5}})(u_{ij}+u_{jj})^{*}=0 \inlineeqnum\label{eqn:idt7}.$ Thus, multiplying (\ref{eqn:idt6}) by the scalar $(\overline{\alpha _{1}+\alpha _{3}+\alpha _{5}}),$ (\ref{eqn:idt7}) by the scalar $(\alpha _{2}+\alpha _{4}+\alpha _{6})$ and subtracting the resulting identities, we arrive at $(|\alpha _{1}+\alpha _{3}+\alpha _{5}|^{2}-|\alpha _{2}+\alpha _{4}+\alpha _{6}|^ {2})(u_{ij}+u_{jj})=0$ which leads to $u_{ij}=0$ and $u_{jj}=0.$ Next, we have 
{\allowdisplaybreaks\begin{align*}
&\Phi (\alpha _{1} 1_{\mathcal{A}}up_{i}+\alpha _{2} 1_{\mathcal{A}}^{*}p_{i}u^{*}+\alpha _{3} u1_{\mathcal{A}}^{*}p_{i}+\alpha _{4} p_{i}1_{\mathcal{A}}u^{*}+\alpha _{5} up_{i}1_{\mathcal{A}}+\alpha _{6} p_{i}u^{*}1_{\mathcal{A}}^{*})\\
&=\alpha _{1} \Phi (1_{\mathcal{A}})\Phi (u)\Phi (p_{i})+\alpha _{2} \Phi (1_{\mathcal{A}})^{*}\Phi (p_{i})\Phi (u)^{*}+\alpha _{3} \Phi (u)\Phi (1_{\mathcal{A}})^{*}\Phi (p_{i})\\
&+\alpha _{4} \Phi (p_{i})\Phi (1_{\mathcal{A}})\Phi (u)^{*}+\alpha _{5} \Phi (u)\Phi (p_{i})\Phi (1_{\mathcal{A}})+\alpha _{6} \Phi (p_{i})\Phi (u)^{*}\Phi (1_{\mathcal{A}})^{*}\\
&=\alpha _{1} \Phi (1_{\mathcal{A}})\Phi (a_{ii}+b_{ij}+c_{ji}+d_{jj})\Phi (p_{i})\\
&+\alpha _{2} \Phi (1_{\mathcal{A}})^{*}\Phi (p_{i})\Phi (a_{ii}+b_{ij}+c_{ji}+d_{jj})^{*}\\
&+\alpha _{3} \Phi (a_{ii}+b_{ij}+c_{ji}+d_{jj})\Phi (1_{\mathcal{A}})^{*}\Phi (p_{i})\\
&+\alpha _{4} \Phi (p_{i})\Phi (1_{\mathcal{A}})\Phi (a_{ii}+b_{ij}+c_{ji}+d_{jj})^{*}\\
&+\alpha _{5} \Phi (a_{ii}+b_{ij}+c_{ji}+d_{jj})\Phi (p_{i})\Phi (1_{\mathcal{A}})\\
&+\alpha _{6} \Phi (p_{i})\Phi (a_{ii}+b_{ij}+c_{ji}+d_{jj})^{*}\Phi (1_{\mathcal{A}})^{*}\\
&-\alpha _{1} \Phi (1_{\mathcal{A}})\Phi (a_{ii}+c_{ji})\Phi (p_{i})-\alpha _{2} \Phi (1_{\mathcal{A}})^{*}\Phi (p_{i})\Phi (a_{ii}+c_{ji})^{*}\\
&-\alpha _{3} \Phi (a_{ii}+c_{ji})\Phi (1_{\mathcal{A}})^{*}\Phi (p_{i})-\alpha _{4} \Phi (p_{i})\Phi (1_{\mathcal{A}})\Phi (a_{ii}+c_{ji})^{*}\\
&-\alpha _{5} \Phi (a_{ii}+c_{ji})\Phi (p_{i})\Phi (1_{\mathcal{A}})-\alpha _{6} \Phi (p_{i})\Phi (a_{ii}+c_{ji})^{*}\Phi (1_{\mathcal{A}})^{*}\\
&-\alpha _{1} \Phi (1_{\mathcal{A}})\Phi (b_{ij}+d_{jj})\Phi (p_{i})-\alpha _{2} \Phi (1_{\mathcal{A}})^{*}\Phi (p_{i})\Phi (b_{ij}+d_{jj})^{*}\\
&-\alpha _{3} \Phi (b_{ij}+d_{jj})\Phi (1_{\mathcal{A}})^{*}\Phi (p_{i})-\alpha _{4} \Phi (p_{i})\Phi (1_{\mathcal{A}})\Phi (b_{ij}+d_{jj})^{*}\\
&-\alpha _{5} \Phi (b_{ij}+d_{jj})\Phi (p_{i})\Phi (1_{\mathcal{A}})-\alpha _{6} \Phi (p_{i})\Phi (b_{ij}+d_{jj})^{*}\Phi (1_{\mathcal{A}})^{*}\\
&=\Phi (\alpha _{1} 1_{\mathcal{A}}(a_{ii}+b_{ij}+c_{ji}+d_{jj})p_{i}+\alpha _{2} 1_{\mathcal{A}}^{*}p_{i}(a_{ii}+b_{ij}+c_{ji}+d_{jj})^{*}\\
&+\alpha _{3} (a_{ii}+b_{ij}+c_{ji}+d_{jj})1_{\mathcal{A}}^{*}p_{i}+\alpha _{4} p_{i}1_{\mathcal{A}}(a_{ii}+b_{ij}+c_{ji}+d_{jj})^{*}\\
&+\alpha _{5} (a_{ii}+b_{ij}+c_{ji}+d_{jj})p_{i}1_{\mathcal{A}}+\alpha _{6} p_{i}(a_{ii}+b_{ij}+c_{ji}+d_{jj})^{*}1_{\mathcal{A}}^{*})\\
&-\Phi (\alpha _{1} 1_{\mathcal{A}}(a_{ii}+c_{ji})p_{i}+\alpha _{2} 1_{\mathcal{A}}^{*}p_{i}(a_{ii}+c_{ji})^{*}+\alpha _{3} (a_{ii}+c_{ji})1_{\mathcal{A}}^{*}p_{i}\\
&+\alpha _{4} p_{i}1_{\mathcal{A}}(a_{ii}+c_{ji})^{*}+\alpha _{5} (a_{ii}+c_{ji})p_{i}1_{\mathcal{A}}+\alpha _{6} p_{i}(a_{ii}+c_{ji})^{*}1_{\mathcal{A}}^{*})\\
&-\Phi (\alpha _{1} 1_{\mathcal{A}}(b_{ij}+d_{jj})p_{i}+\alpha _{2} 1_{\mathcal{A}}^{*}p_{i}(b_{ij}+d_{jj})^{*}+\alpha _{3} (b_{ij}+d_{jj})1_{\mathcal{A}}^{*}p_{i}\\
&+\alpha _{4} p_{i}1_{\mathcal{A}}(b_{ij}+d_{jj})^{*}+\alpha _{5} (b_{ij}+d_{jj})p_{i}1_{\mathcal{A}}+\alpha _{6} p_{i}(b_{ij}+d_{jj})^{*}1_{\mathcal{A}}^{*})\\
&=0
\end{align*}}
from which we immediately deduce the identity $\alpha _{1} 1_{\mathcal{A}}up_{i}+\alpha _{2} 1_{\mathcal{A}}^{*}p_{i}u^{*}+\alpha _{3} u1_{\mathcal{A}}^{*}p_{i}+\alpha _{4} p_{i}1_{\mathcal{A}}u^{*}+\alpha _{5} up_{i}1_{\mathcal{A}}+\alpha _{6} p_{i}u^{*}1_{\mathcal{A}}^{*}=0.$ This results in the identity $(\alpha _{1} +\alpha _{3} +\alpha _{5})(u_{ii}+u_{ji})+(\alpha _{2} +\alpha _{4} +\alpha _{6})(u_{ii}+u_{ji})^{*}=0 \inlineeqnum\label{eqn:idt8}.$ Also, we get $(\overline{\alpha _{2} +\alpha _{4} +\alpha _{6}})(u_{ii}+u_{ji})+(\overline{\alpha _{1} +\alpha _{3} +\alpha _{5}})(u_{ii}+u_{ji})^{*}=0 \inlineeqnum\label{eqn:idt9},$ by the application of the involution $\ast $ on (\ref{eqn:idt8}). As a consequence, multiplying (\ref{eqn:idt8}) by the scalar $(\overline{\alpha _{1}+\alpha _{3}+\alpha _{5}}),$ (\ref{eqn:idt9}) by the scalar $(\alpha _{2}+\alpha _{4}+\alpha _{6})$ and subtracting the resulting identities, we arrive at $(|\alpha _{1}+\alpha _{3}+\alpha _{5}|^{2}-|\alpha _{2}+\alpha _{4}+\alpha _{6}|^ {2})(u_{ii}+u_{ji})=0$ which shows that $u_{ii}+u_{ji}=0.$ Consequently, we obtain $u_{ii}=0$ and $u_{ji}=0.$
\end{proof}

\begin{claim}\label{s24} For every $a_{ij},b_{ij}\in \mathcal{A}_{ij}$ $(i\neq j;i,j=1,2)$ we have: $\Phi (a_{ij}+b_{ij})=\Phi (a_{ij})+ \Phi (b_{ij}).$
\end{claim}
\begin{proof} First, note that the following identity holds
{\allowdisplaybreaks\begin{align*}
&(\alpha _{1}+\alpha _{3}+\alpha _{5})(a_{ij}+b_{ij})+(\alpha _{2}+\alpha _{4}+\alpha _{6})(a_{ij}^{*}+b_{ij}a_{ij}^{*})\\
&=\alpha _{1} 1_{\mathcal{A}}(p_{i}+a_{ij})(p_{j}+b_{ij})+\alpha _{2} 1_{\mathcal{A}}^{*}(p_{j}+b_{ij})(p_{i}+a_{ij})^{*}\\
&+\alpha _{3} (p_{i}+a_{ij})1_{\mathcal{A}}^{*}(p_{j}+b_{ij})+\alpha _{4} (p_{j}+b_{ij})1_{\mathcal{A}}(p_{i}+a_{ij})^{*}\\
&+\alpha _{5} (p_{j}+a_{ij})(p_{i}+b_{ij})1_{\mathcal{A}}+\alpha _{6} (p_{j}+b_{ij})(p_{i}+a_{ij})^{*}1_{\mathcal{A}}^{*},
\end{align*}}
for all elements $a_{ij},b_{ij}\in \mathcal{A}_{ij}.$ Hence, by Claim \ref{s23} we have
{\allowdisplaybreaks\begin{align*}
&\Phi ((\alpha _{1}+\alpha _{3}+\alpha _{5}) (a_{ij}+b_{ij}))+\Phi ((\alpha _{2}+\alpha _{4}+\alpha _{6})(a_{ij}^{*}+b_{ij}a_{ij}^{*}))\\
&=\Phi ((\alpha _{1}+\alpha _{3}+\alpha _{5}) (a_{ij}+b_{ij})+(\alpha _{2}+\alpha _{4}+\alpha _{6})(a_{ij}^{*}+b_{ij}a_{ij}^{*}))\\
&=\Phi (\alpha _{1} 1_{\mathcal{A}}(p_{i}+a_{ij})(p_{j}+b_{ij})+\alpha _{2} 1_{\mathcal{A}}^{*}(p_{j}+b_{ij})(p_{i}+a_{ij})^{*}\\
&+\alpha _{3} (p_{i}+a_{ij})1_{\mathcal{A}}^{*}(p_{j}+b_{ij})+\alpha _{4} (p_{j}+b_{ij})1_{\mathcal{A}}(p_{i}+a_{ij})^{*}\\
&+\alpha _{5} (p_{j}+a_{ij})(p_{i}+b_{ij})1_{\mathcal{A}}+\alpha _{6} (p_{j}+b_{ij})(p_{i}+a_{ij})^{*}1_{\mathcal{A}}^{*})\\
&=\alpha _{1} \Phi (1_{\mathcal{A}})\Phi (p_{i}+a_{ij})\Phi (p_{j}+b_{ij})+\alpha _{2} \Phi (1_{\mathcal{A}})^{*}\Phi (p_{j}+b_{ij})\Phi (p_{i}+a_{ij})^{*}\\
&+\alpha _{3} \Phi (p_{i}+a_{ij})\Phi (1_{\mathcal{A}})^{*}\Phi (p_{j}+b_{ij})+\alpha _{4} \Phi (p_{j}+b_{ij})\Phi (1_{\mathcal{A}})\Phi (p_{i}+a_{ij})^{*}\\
&+\alpha _{5} \Phi (p_{i}+a_{ij})\Phi (p_{j}+b_{ij})\Phi (1_{\mathcal{A}})+\alpha _{6} \Phi (p_{j}+b_{ij})\Phi (p_{i}+a_{ij})^{*}\Phi (1_{\mathcal{A}})^{*}\\
&=\alpha _{1} \Phi (1_{\mathcal{A}})(\Phi (p_{i})+\Phi (a_{ij}))(\Phi (p_{j})+\Phi (b_{ij}))\\
&+\alpha _{2} \Phi (1_{\mathcal{A}})^{*}(\Phi (p_{j})+\Phi (b_{ij}))(\Phi (p_{i})^{*}+\Phi (a_{ij})^{*})\\
&+\alpha _{3} (\Phi (p_{i})+\Phi (a_{ij}))\Phi (1_{\mathcal{A}})^{*}(\Phi (p_{j})+\Phi (b_{ij}))\\
&+\alpha _{4} (\Phi (p_{j})+\Phi (b_{ij}))\Phi (1_{\mathcal{A}})(\Phi (p_{i})^{*}+\Phi (a_{ij})^{*})\\
&+\alpha _{5} (\Phi (p_{i})+\Phi (a_{ij}))(\Phi (p_{j})+\Phi (b_{ij}))\Phi (1_{\mathcal{A}})\\
&+\alpha _{6} (\Phi (p_{j})+\Phi (b_{ij}))(\Phi (p_{i})^{*}+\Phi (a_{ij})^{*})\Phi (1_{\mathcal{A}})^{*}\\
&=\alpha _{1} \Phi (1_{\mathcal{A}})\Phi (p_{i})\Phi (p_{j})+\alpha _{2} \Phi (1_{\mathcal{A}})^{*}\Phi (p_{j})\Phi (p_{i})^{*}+\alpha _{3} \Phi (p_{i})\Phi (1_{\mathcal{A}})^{*}\Phi (p_{j})\\
&+\alpha _{4} \Phi (p_{j})\Phi (1_{\mathcal{A}})\Phi (p_{i})^{*}+\alpha _{5} \Phi (p_{i})\Phi (p_{j})\Phi (1_{\mathcal{A}})+\alpha _{6} \Phi (p_{j})\Phi (p_{i})^{*}\Phi (1_{\mathcal{A}})^{*}\\
&+\alpha _{1} \Phi (1_{\mathcal{A}})\Phi (a_{ij})\Phi (p_{j})+\alpha _{2} \Phi (1_{\mathcal{A}})^{*}\Phi (p_{j})\Phi (a_{ij})^{*}+\alpha _{3} \Phi (a_{ij})\Phi (1_{\mathcal{A}})^{*}\Phi (p_{j})\\
&+\alpha _{4} \Phi (p_{j})\Phi (1_{\mathcal{A}})\Phi (a_{ij})^{*}+\alpha _{5} \Phi (a_{ij})\Phi (p_{j})\Phi (1_{\mathcal{A}})+\alpha _{6} \Phi (p_{j})\Phi (a_{ij})^{*}\Phi (1_{\mathcal{A}})^{*}\\
&+\alpha _{1} \Phi (1_{\mathcal{A}})\Phi (p_{i})\Phi (b_{ij})
+\alpha _{2} \Phi (1_{\mathcal{A}})^{*}\Phi (b_{ij})\Phi (p_{i})^{*}
+\alpha _{3} \Phi (p_{i})\Phi (1_{\mathcal{A}})^{*}\Phi (b_{ij})\\
&+\alpha _{4} \Phi (b_{ij})\Phi (1_{\mathcal{A}})\Phi (p_{i})^{*}
+\alpha _{5} \Phi (p_{i})\Phi (b_{ij})\Phi (1_{\mathcal{A}})
+\alpha _{6} \Phi (b_{ij})\Phi (p_{i})^{*}\Phi (1_{\mathcal{A}})^{*}\\
&+\alpha _{1} \Phi (1_{\mathcal{A}})\Phi (a_{ij})\Phi (b_{ij})
+\alpha _{2} \Phi (1_{\mathcal{A}})^{*}\Phi (b_{ij})\Phi (a_{ij})^{*}
+\alpha _{3} \Phi (a_{ij})\Phi (1_{\mathcal{A}})^{*}\Phi (b_{ij})\\
&+\alpha _{4} \Phi (b_{ij})\Phi (1_{\mathcal{A}})\Phi (a_{ij})^{*}
+\alpha _{5} \Phi (a_{ij})\Phi (b_{ij})\Phi (1_{\mathcal{A}})
+\alpha _{6} \Phi (b_{ij})\Phi (a_{ij})^{*}\Phi (1_{\mathcal{A}})^{*}\\
&=\Phi (\alpha _{1} 1_{\mathcal{A}}p_{i}p_{j}+\alpha _{2} 1_{\mathcal{A}}^{*}p_{j}p_{i}^{*}+\alpha _{3} p_{i}1_{\mathcal{A}}^{*}p_{j}+\alpha _{4} p_{j}1_{\mathcal{A}}p_{i}^{*}+\alpha _{5} p_{i}p_{j}1_{\mathcal{A}}\\
&+\alpha _{6} p_{j}p_{i}^{*}1_{\mathcal{A}}^{*})+\Phi (\alpha _{1} 1_{\mathcal{A}}a_{ij}p_{j}+\alpha _{2} 1_{\mathcal{A}}^{*}p_{j}a_{ij}^{*}+\alpha _{3} a_{ij}1_{\mathcal{A}}^{*}p_{j}+\alpha _{4} p_{j}1_{\mathcal{A}}a_{ij}^{*}\\
&+\alpha _{5} a_{ij}p_{j}1_{\mathcal{A}}+\alpha _{6} p_{j}a_{ij}^{*}1_{\mathcal{A}}^{*})+\Phi (\alpha _{1} 1_{\mathcal{A}}p_{i}b_{ij}
+\alpha _{2}1_{\mathcal{A}}^{*}b_{ij}p_{i}^{*}
+\alpha _{3} p_{i}1_{\mathcal{A}}^{*}b_{ij}\\
&+\alpha _{4} b_{ij}1_{\mathcal{A}}p_{i}^{*}+\alpha _{5} p_{i}b_{ij}1_{\mathcal{A}}
+\alpha _{6} b_{ij}p_{i}^{*}1_{\mathcal{A}}^{*})+\Phi (\alpha _{1} 1_{\mathcal{A}}a_{ij}b_{ij}+\alpha _{2} 1_{\mathcal{A}}^{*}b_{ij}a_{ij}^{*}\\
&+\alpha _{3} a_{ij}1_{\mathcal{A}}^{*}b_{ij}+\alpha _{4} b_{ij}1_{\mathcal{A}}a_{ij}^{*}+\alpha _{5} a_{ij}b_{ij}1_{\mathcal{A}}+\alpha _{6} b_{ij}a_{ij}^{*}1_{\mathcal{A}}^{*})\\
&=\Phi ((\alpha _{1}+\alpha _{3}+\alpha _{5}) a_{ij})+\Phi ((\alpha _{2}+\alpha _{4}+\alpha _{6})a_{ij}^{*})+\Phi ((\alpha _{1}+\alpha _{3}+\alpha _{5}) b_{ij})\\
&+\Phi ((\alpha _{2}+\alpha _{4}+\alpha _{6})b_{ij}a_{ij}^{*}).
\end{align*}}
As a consequence we get $\Phi ((\alpha _{1}+\alpha _{3}+\alpha _{5}) (a_{ij}+b_{ij}))=\Phi ((\alpha _{1}+\alpha _{3}+\alpha _{5}) a_{ij})+\Phi ((\alpha _{1}+\alpha _{3}+\alpha _{5}) b_{ij}).$ Next, note that the following identity holds
{\allowdisplaybreaks\begin{align*}
&(\alpha _{2}+\alpha _{4}+\alpha _{6})(a_{ij}+b_{ij})+(\alpha _{1}+\alpha _{3}+\alpha _{5})(b_{ij}^{*}+b_{ij}^{*}a_{ij})\\
&=\alpha _{1} 1_{\mathcal{A}}(p_{j}+b_{ij}^{*})(p_{i}+a_{ij})+\alpha _{2} 1_{\mathcal{A}}^{*}(p_{i}+a_{ij})(p_{j}+b_{ij}^{*})^{*}\\
&+\alpha _{3} (p_{j}+b_{ij}^{*})1_{\mathcal{A}}^{*}(p_{i}+a_{ij})+\alpha _{4} (p_{i}+a_{ij})1_{\mathcal{A}}(p_{j}+b_{ij}^{*})^{*}\\
&+\alpha _{5} (p_{i}+b_{ij}^{*})(p_{j}+a_{ij})1_{\mathcal{A}}+\alpha _{6} (p_{i}+a_{ij})(p_{j}+b_{ij}^{*})^{*}1_{\mathcal{A}}^{*},
\end{align*}}
for all elements $a_{ij},b_{ij}\in \mathcal{A}_{ij}.$ Using a similar process to the first case above, we get $\Phi ((\alpha _{2}+\alpha _{4}+\alpha _{6}) (a_{ij}+b_{ij}))=\Phi ((\alpha _{2}+\alpha _{4}+\alpha _{6}) a_{ij})+\Phi ((\alpha _{2}+\alpha _{4}+\alpha _{6}) b_{ij}).$ The  two last obtained results lead to the conclusion that $\Phi (a_{ij}+b_{ij})=\Phi (a_{ij})+ \Phi (b_{ij}).$
\end{proof}

\begin{claim}\label{s25} For every $a_{ii},b_{ii}\in \mathcal{A}_{ii}$ $(i=1,2),$ we have: $\Phi (a_{ii}+b_{ii})=\Phi (a_{ii})+\Phi (b_{ii}).$ 
\end{claim}
\begin{proof} Let $u=u_{ii}+u_{ij}+u_{ji}+u_{jj}=\Phi ^{-1}(\Phi (a_{ii}+b_{ii})-\Phi (a_{ii})-\Phi (b_{ii}))\in \mathcal{A}$ $(i\neq j; i,j=1,2).$ Then {\allowdisplaybreaks\begin{align*}
&\Phi (\alpha _{1} 1_{\mathcal{A}}up_{j}+\alpha _{2} 1_{\mathcal{A}}^{*}p_{j}u^{*}+\alpha _{3} u1_{\mathcal{A}}^{*}p_{j}+\alpha _{4} p_{j}1_{\mathcal{A}}u^{*}+\alpha _{5} up_{j}1_{\mathcal{A}}+\alpha _{6} p_{j}u^{*}1_{\mathcal{A}}^{*})\\
&=\alpha _{1} \Phi (1_{\mathcal{A}})\Phi (u)\Phi (p_{j})+\alpha _{2} \Phi (1_{\mathcal{A}})^{*}\Phi (p_{j})\Phi (u)^{*}+\alpha _{3} \Phi (u)\Phi (1_{\mathcal{A}})^{*}\Phi (p_{j})\\
&+\alpha _{4} \Phi (p_{j})\Phi (1_{\mathcal{A}})\Phi (u)^{*}+\alpha _{5} \Phi (u)\Phi (p_{j})\Phi (1_{\mathcal{A}})+\alpha _{6} \Phi (p_{j})\Phi (u)^{*}\Phi (1_{\mathcal{A}})^{*}\\
&=\alpha _{1} \Phi (1_{\mathcal{A}})\Phi (a_{ii}+b_{ii})\Phi (p_{j})+\alpha _{2} \Phi (1_{\mathcal{A}})^{*}\Phi (p_{j})\Phi (a_{ii}+b_{ii})^{*}\\
&+\alpha _{3} \Phi (a_{ii}+b_{ii})\Phi (1_{\mathcal{A}})^{*}\Phi (p_{j})+\alpha _{4} \Phi (p_{j})\Phi (1_{\mathcal{A}})\Phi (a_{ii}+b_{ii})^{*}\\
&+\alpha _{5} \Phi (a_{ii}+b_{ii})\Phi (p_{j})\Phi (1_{\mathcal{A}})+\alpha _{6} \Phi (p_{j})\Phi (a_{ii}+b_{ii})^{*}\Phi (1_{\mathcal{A}})^{*}\\
&-\alpha _{1} \Phi (1_{\mathcal{A}})\Phi (a_{ii})\Phi (p_{j})-\alpha _{2} \Phi (1_{\mathcal{A}})^{*}\Phi (p_{j})\Phi (a_{ii})^{*}\\
&-\alpha _{3} \Phi (a_{ii})\Phi (1_{\mathcal{A}})^{*}\Phi (p_{j})-\alpha _{4} \Phi (p_{j})\Phi (1_{\mathcal{A}})\Phi (a_{ii})^{*}\\
&-\alpha _{5} \Phi (a_{ii})\Phi (p_{j})\Phi (1_{\mathcal{A}})-\alpha _{6} \Phi (p_{j})\Phi (a_{ii})^{*}\Phi (1_{\mathcal{A}})^{*}\\
&-\alpha _{1} \Phi (1_{\mathcal{A}})\Phi (b_{ii})\Phi (p_{j})-\alpha _{2} \Phi (1_{\mathcal{A}})^{*}\Phi (p_{j})\Phi (b_{ii})^{*}\\
&-\alpha _{3} \Phi (b_{ii})\Phi (1_{\mathcal{A}})^{*}\Phi (p_{j})-\alpha _{4} \Phi (p_{j})\Phi (1_{\mathcal{A}})\Phi (b_{ii})^{*}\\
&-\alpha _{5} \Phi (b_{ii})\Phi (p_{j})\Phi (1_{\mathcal{A}})-\alpha _{6} \Phi (p_{j})\Phi (b_{ii})^{*}\Phi (1_{\mathcal{A}})^{*}\\
&=\Phi (\alpha _{1} 1_{\mathcal{A}}(a_{ii}+b_{ii})p_{j}+\alpha _{2} 1_{\mathcal{A}}^{*}p_{j}(a_{ii}+b_{ii})^{*}+\alpha _{3} (a_{ii}+b_{ii})1_{\mathcal{A}}^{*}p_{j}\\
&+\alpha _{4} p_{j}1_{\mathcal{A}}(a_{ii}+b_{ii})^{*}+\alpha _{5} (a_{ii}+b_{ii})p_{j}1_{\mathcal{A}}+\alpha _{6} p_{j}(a_{ii}+b_{ii})^{*}1_{\mathcal{A}}^{*})\\
&-\Phi (\alpha _{1} 1_{\mathcal{A}}a_{ii}p_{j}+\alpha _{2} 1_{\mathcal{A}}^{*}p_{j}a_{ii}^{*}+\alpha _{3} a_{ii}1_{\mathcal{A}}^{*}p_{j}+\alpha _{4} p_{j}1_{\mathcal{A}}a_{ii}^{*}+\alpha _{5} a_{ii}p_{j}1_{\mathcal{A}}\\
&+\alpha _{6} p_{j}a_{ii}^{*}1_{\mathcal{A}}^{*})-\Phi (\alpha _{1} 1_{\mathcal{A}}b_{ii}p_{j}+\alpha _{2} 1_{\mathcal{A}}^{*}p_{j}b_{ii}^{*}+\alpha _{3} b_{ii}1_{\mathcal{A}}^{*}p_{j}+\alpha _{4} p_{j}1_{\mathcal{A}}b_{ii}^{*}\\
&+\alpha _{5} b_{ii}p_{j}1_{\mathcal{A}}+\alpha _{6} p_{j}b_{ii}^{*}1_{\mathcal{A}}^{*})=0
\end{align*}}
which leads directly to the identity $\alpha _{1} 1_{\mathcal{A}}up_{j}+\alpha _{2} 1_{\mathcal{A}}^{*}p_{j}u^{*}+\alpha _{3} u1_{\mathcal{A}}^{*}p_{j}+\alpha _{4} p_{j}1_{\mathcal{A}}u^{*}+\alpha _{5} up_{j}1_{\mathcal{A}}+\alpha _{6} p_{j}u^{*}1_{\mathcal{A}}^{*}=0.$ It therefore follows that $(\alpha _{1}+\alpha _{3}+\alpha _{5})u_{ij}+(\alpha _{2}+\alpha _{4}+\alpha _{6})u_{ij}^{*}+(\alpha _{1}+\alpha _{3}+\alpha _{5})u_{jj}+(\alpha _{2}+\alpha _{4}+\alpha _{6})u_{jj}^{*}=0 \inlineeqnum\label{eqn:idt10}$ and hence the identity $(\overline{\alpha _{2}+\alpha _{4}+\alpha _{6}})u_{ij}+(\overline{\alpha _{1}+\alpha _{3}+\alpha _{5}})u_{ij}^{*}+(\overline{\alpha _{2}+\alpha _{4}+\alpha _{6}})u_{jj}+(\overline{\alpha _{1}+\alpha _{3}+\alpha _{5}})u_{jj}^{*}=0 \inlineeqnum\label{eqn:idt11}.$ From (\ref{eqn:idt10}) and (\ref{eqn:idt11}), we get $(|\alpha _{1}+\alpha _{3}+\alpha _{5}|^{2}-|\alpha _{2}+\alpha _{4}+\alpha _{6}|^ {2})(u_{ij}+u_{jj})=0$ which implies that $u_{ij}+u_{jj}=0.$ This results that $u_{ij}=0$ and $u_{jj}=0.$ Next, for all element $t_{ij}\in \mathcal{A}_{ij}$ we have
{\allowdisplaybreaks\begin{align*}
&\Phi (\alpha _{1} 1_{\mathcal{A}}t_{ij}u+\alpha _{2} 1_{\mathcal{A}}^{*}ut_{ij}^{*}+\alpha _{3} t_{ij}1_{\mathcal{A}}^{*}u+\alpha _{4} u1_{\mathcal{A}}t_{ij}^{*}+\alpha _{5} t_{ij}u1_{\mathcal{A}}+\alpha _{6} ut_{ij}^{*}1_{\mathcal{A}}^{*})\\
&=\alpha _{1} \Phi (1_{\mathcal{A}})\Phi (t_{ij})\Phi (u)+\alpha _{2} \Phi (1_{\mathcal{A}})^{*}\Phi (u)\Phi (t_{ij})^{*}+\alpha _{3} \Phi (t_{ij})\Phi (1_{\mathcal{A}})^{*}\Phi (u)\\
&+\alpha _{4} \Phi (u)\Phi (1_{\mathcal{A}})\Phi (t_{ij})^{*}+\alpha _{5} \Phi (t_{ij})\Phi (u)\Phi (1_{\mathcal{A}})+\alpha _{6} \Phi (u)\Phi (t_{ij})^{*}\Phi (1_{\mathcal{A}})^{*}\\
&=\alpha _{1} \Phi (1_{\mathcal{A}})\Phi (t_{ij})\Phi (a_{ii}+b_{ii})+\alpha _{2} \Phi (1_{\mathcal{A}})^{*}\Phi (a_{ii}+b_{ii})\Phi (t_{ij})^{*}\\
&+\alpha _{3} \Phi (t_{ij})\Phi (1_{\mathcal{A}})^{*}\Phi (a_{ii}+b_{ii})+\alpha _{4} \Phi (a_{ii}+b_{ii})\Phi (1_{\mathcal{A}})\Phi (t_{ij})^{*}\\
&+\alpha _{5} \Phi (t_{ij})\Phi (a_{ii}+b_{ii})\Phi (1_{\mathcal{A}})+\alpha _{6} \Phi (a_{ii}+b_{ii})\Phi (t_{ij})^{*}\Phi (1_{\mathcal{A}})^{*}\\
&-\alpha _{1} \Phi (1_{\mathcal{A}})\Phi (t_{ij})\Phi (a_{ii})-\alpha _{2} \Phi (1_{\mathcal{A}})^{*}\Phi (a_{ii})\Phi (t_{ij})^{*}-\alpha _{3} \Phi (t_{ij})\Phi (1_{\mathcal{A}})^{*}\Phi (a_{ii})\\
&-\alpha _{4} \Phi (a_{ii})\Phi (1_{\mathcal{A}})\Phi (t_{ij})^{*}-\alpha _{5} \Phi (t_{ij})\Phi (a_{ii})\Phi (1_{\mathcal{A}})-\alpha _{6} \Phi (a_{ii})\Phi (t_{ij})^{*}\Phi (1_{\mathcal{A}})^{*}\\
&-\alpha _{1} \Phi (1_{\mathcal{A}})\Phi (t_{ij})\Phi (b_{ii})-\alpha _{2} \Phi (1_{\mathcal{A}})^{*}\Phi (b_{ii})\Phi (t_{ij})^{*}-\alpha _{3} \Phi (t_{ij})\Phi (1_{\mathcal{A}})^{*}\Phi (b_{ii})\\
&-\alpha _{4} \Phi (b_{ii})\Phi (1_{\mathcal{A}})\Phi (t_{ij})^{*}-\alpha _{5} \Phi (t_{ij})\Phi (b_{ii})\Phi (1_{\mathcal{A}})-\alpha _{6} \Phi (b_{ii})\Phi (t_{ij})^{*}\Phi (1_{\mathcal{A}})^{*}\\
&=\Phi (\alpha _{1} 1_{\mathcal{A}}t_{ij}(a_{ii}+b_{ii})+\alpha _{2} 1_{\mathcal{A}}^{*}(a_{ii}+b_{ii})t_{ij}^{*}+\alpha _{3} t_{ij}1_{\mathcal{A}}^{*}(a_{ii}+b_{ii})\\
&+\alpha _{4} (a_{ii}+b_{ii})1_{\mathcal{A}}t_{ij}^{*}+\alpha _{5} t_{ij}(a_{ii}+b_{ii})1_{\mathcal{A}}+\alpha _{6} (a_{ii}+b_{ii})t_{ij}^{*}1_{\mathcal{A}}^{*})\\
&-\Phi (\alpha _{1} 1_{\mathcal{A}}t_{ij}a_{ii}+\alpha _{2} 1_{\mathcal{A}}^{*}a_{ii}t_{ij}^{*}+\alpha _{3} t_{ij}1_{\mathcal{A}}^{*}a_{ii} +\alpha _{4} a_{ii}1_{\mathcal{A}}t_{ij}^{*}+\alpha _{5} t_{ij}a_{ii}1_{\mathcal{A}}\\
&+\alpha _{6} a_{ii}t_{ij}^{*}1_{\mathcal{A}}^{*})-\Phi (\alpha _{1} 1_{\mathcal{A}}t_{ij}b_{ii}+\alpha _{2} 1_{\mathcal{A}}^{*}b_{ii}t_{ij}^{*}+\alpha _{3} t_{ij}1_{\mathcal{A}}^{*}b_{ii}+\alpha _{4} b_{ii}1_{\mathcal{A}}t_{ij}^{*}\\
&+\alpha _{5} t_{ij}b_{ii}1_{\mathcal{A}}+\alpha _{6} b_{ii}t_{ij}^{*}1_{\mathcal{A}}^{*})=0.
\end{align*}}
It follows immediately from this that $\alpha _{1} 1_{\mathcal{A}}t_{ij}u+\alpha _{2} 1_{\mathcal{A}}^{*}ut_{ij}^{*}+\alpha _{3} t_{ij}1_{\mathcal{A}}^{*}u+\alpha _{4} u1_{\mathcal{A}}t_{ij}^{*}+\alpha _{5} t_{ij}u1_{\mathcal{A}}+\alpha _{6} ut_{ij}^{*}1_{\mathcal{A}}^{*}=0.$ From this identity, we can deduce that $(\alpha _{1} +\alpha _{3} +\alpha _{5})t_{ij}u_{ji}=0$ which leads to $(\alpha _{1} +\alpha _{3} +\alpha _{5})u_{ji}=0 \inlineeqnum\label{eqn:idt12},$ because of the primeness of $\mathcal{A}.$ Also, for all element $t_{ji}\in \mathcal{A}_{ji}$ we have
{\allowdisplaybreaks\begin{align*}
&\Phi (\alpha _{1} 1_{\mathcal{A}}t_{ji}u+\alpha _{2} 1_{\mathcal{A}}^{*}ut_{ji}^{*}+\alpha _{3} t_{ji}1_{\mathcal{A}}^{*}u+\alpha _{4} u1_{\mathcal{A}}t_{ji}^{*}+\alpha _{5} t_{ji}u1_{\mathcal{A}}+\alpha _{6} ut_{ji}^{*}1_{\mathcal{A}}^{*})\\
&=\alpha _{1} \Phi (1_{\mathcal{A}})\Phi (t_{ji})\Phi (u)+\alpha _{2} \Phi (1_{\mathcal{A}})^{*}\Phi (u)\Phi (t_{ji})^{*}+\alpha _{3} \Phi (t_{ji})\Phi (1_{\mathcal{A}})^{*}\Phi (u)\\
&+\alpha _{4} \Phi (u)\Phi (1_{\mathcal{A}})\Phi (t_{ji})^{*}+\alpha _{5} \Phi (t_{ji})\Phi (u)\Phi (1_{\mathcal{A}})+\alpha _{6} \Phi (u)\Phi (t_{ji})^{*}\Phi (1_{\mathcal{A}})^{*}\\
&=\alpha _{1} \Phi (1_{\mathcal{A}})\Phi (t_{ji})\Phi (a_{ii}+b_{ii})+\alpha _{2} \Phi (1_{\mathcal{A}})^{*}\Phi (a_{ii}+b_{ii})\Phi (t_{ji})^{*}\\
&+\alpha _{3} \Phi (t_{ji})\Phi (1_{\mathcal{A}})^{*}\Phi (a_{ii}+b_{ii})+\alpha _{4} \Phi (a_{ii}+b_{ii})\Phi (1_{\mathcal{A}})\Phi (t_{ji})^{*}\\
&+\alpha _{5} \Phi (t_{ji})\Phi (a_{ii}+b_{ii})\Phi (1_{\mathcal{A}})+\alpha _{6} \Phi (a_{ii}+b_{ii})\Phi (t_{ji})^{*}\Phi (1_{\mathcal{A}})^{*}\\
&-\alpha _{1} \Phi (1_{\mathcal{A}})\Phi (t_{ji})\Phi (a_{ii})-\alpha _{2} \Phi (1_{\mathcal{A}})^{*}\Phi (a_{ii})\Phi (t_{ji})^{*}-\alpha _{3} \Phi (t_{ji})\Phi (1_{\mathcal{A}})^{*}\Phi (a_{ii})\\
&-\alpha _{4} \Phi (a_{ii})\Phi (1_{\mathcal{A}})\Phi (t_{ji})^{*}-\alpha _{5} \Phi (t_{ji})\Phi (a_{ii})\Phi (1_{\mathcal{A}})-\alpha _{6} \Phi (a_{ii})\Phi (t_{ji})^{*}\Phi (1_{\mathcal{A}})^{*}\\
&-\alpha _{1} \Phi (1_{\mathcal{A}})\Phi (t_{ji})\Phi (b_{ii})-\alpha _{2} \Phi (1_{\mathcal{A}})^{*}\Phi (b_{ii})\Phi (t_{ji})^{*}-\alpha _{3} \Phi (t_{ji})\Phi (1_{\mathcal{A}})^{*}\Phi (b_{ii})\\
&-\alpha _{4} \Phi (b_{ii})\Phi (1_{\mathcal{A}})\Phi (t_{ji})^{*}-\alpha _{5} \Phi (t_{ji})\Phi (b_{ii})\Phi (1_{\mathcal{A}})-\alpha _{6} \Phi (b_{ii})\Phi (t_{ji})^{*}\Phi (1_{\mathcal{A}})^{*}\\
&=\Phi (\alpha _{1} 1_{\mathcal{A}}t_{ji}(a_{ii}+b_{ii})+\alpha _{2} 1_{\mathcal{A}}^{*}(a_{ii}+b_{ii})t_{ji}^{*}+\alpha _{3} t_{ji}1_{\mathcal{A}}^{*}(a_{ii}+b_{ii})\\
&+\alpha _{4} (a_{ii}+b_{ii})1_{\mathcal{A}}t_{ji}^{*}+\alpha _{5} t_{ji}(a_{ii}+b_{ii})1_{\mathcal{A}}+\alpha _{6} (a_{ii}+b_{ii})t_{ji}^{*}1_{\mathcal{A}}^{*})\\
&-\Phi (\alpha _{1} 1_{\mathcal{A}}t_{ji}a_{ii}+\alpha _{2} 1_{\mathcal{A}}^{*}a_{ii}t_{ji}^{*}+\alpha _{3} t_{ji}1_{\mathcal{A}}^{*}a_{ii} +\alpha _{4} a_{ii}1_{\mathcal{A}}t_{ji}^{*}+\alpha _{5} t_{ji}a_{ii}1_{\mathcal{A}}\\
&+\alpha _{6} a_{ii}t_{ji}^{*}1_{\mathcal{A}}^{*})-\Phi (\alpha _{1} 1_{\mathcal{A}}t_{ji}b_{ii}+\alpha _{2} 1_{\mathcal{A}}^{*}b_{ii}t_{ji}^{*}+\alpha _{3} t_{ji}1_{\mathcal{A}}^{*}b_{ii}+\alpha _{4} b_{ii}1_{\mathcal{A}}t_{ji}^{*}\\
&+\alpha _{5} t_{ji}b_{ii}1_{\mathcal{A}}+\alpha _{6} b_{ii}t_{ji}^{*}1_{\mathcal{A}}^{*})=0
\end{align*}} which results in the identity $\alpha _{1} 1_{\mathcal{A}}t_{ji}u+\alpha _{2} 1_{\mathcal{A}}^{*}ut_{ji}^{*}+\alpha _{3} t_{ji}1_{\mathcal{A}}^{*}u+\alpha _{4} u1_{\mathcal{A}}t_{ji}^{*}+\alpha _{5} t_{ji}u1_{\mathcal{A}}+\alpha _{6} ut_{ji}^{*}1_{\mathcal{A}}^{*}=0.$ This shows that $(\alpha _{1} +\alpha _{3} +\alpha _{5})t_{ji}u_{ii}+(\alpha _{2} +\alpha _{4} +\alpha _{6})u_{ii}t_{ji}^{*}+(\alpha _{2}+\alpha _{4}+\alpha _{6})u_{ji}t_{ji}^{*}=0$ which allows to obtain the identities $(\alpha _{2}+\alpha _{4}+\alpha _{6})u_{ji}t_{ji}^{*}=0,$ $(\alpha _{1} +\alpha _{3} +\alpha _{5})t_{ji}u_{ii}=0$ and $(\alpha _{2} +\alpha _{4} +\alpha _{6})u_{ii}t_{ji}^{*}=0.$ As a consequence we get the identities $(\alpha _{2}+\alpha _{4}+\alpha _{6})u_{ji}=0 \inlineeqnum\label{eqn:idt13},$ $(\alpha _{1} +\alpha _{3} +\alpha _{5})u_{ii}=0 \inlineeqnum\label{eqn:idt14}$ and $(\alpha _{2} +\alpha _{4} +\alpha _{6})u_{ii}=0 \inlineeqnum\label{eqn:idt15},$ respectively. It follows from (\ref{eqn:idt12}) and (\ref{eqn:idt13}) that $(|\alpha _{1}+\alpha _{3}+\alpha _{5}|^{2}-|\alpha _{2}+\alpha _{4}+\alpha _{6}|^ {2})u_{ji}=0$ and from (\ref{eqn:idt14}) and (\ref{eqn:idt15}) that $(|\alpha _{1}+\alpha _{3}+\alpha _{5}|^{2}-|\alpha _{2}+\alpha _{4}+\alpha _{6}|^ {2})u_{ii}=0.$ Therefore, we must have $u_{ji}=0$ and $u_{ii}=0.$
\end{proof}

\begin{claim}\label{s26} $\Phi $ is an additive map.
\end{claim}
\begin{proof} The result is a direct consequence of Claims \ref{s23}, \ref{s24} and \ref{s25}.
\end{proof}

In what follows, we prove the second part of the Theorem \ref{thm11}. In the remainder of this paper, all Claims satisfy the conditions (i)-(ii).

\begin{claim}\label{s27} (i) $\Phi (1_{\mathcal{A}})=1_{\mathcal{B}}$ and (ii) $\Phi ((\textstyle \sum _{k=1}^{6} \alpha _{k})c)=(\textstyle \sum _{k=1}^{6} \alpha _{k})\Phi (c),$ for all element $c\in \mathcal{A}.$ 
\end{claim}
\begin{proof} First, note that
{\allowdisplaybreaks\begin{align*}
&\Phi ((\textstyle \sum _{k=1}^{6} \alpha _{k})1_{\mathcal{A}})=\Phi (\alpha _{1} 1_{\mathcal{A}}1_{\mathcal{A}}1_{\mathcal{A}}+\alpha _{2} 1_{\mathcal{A}}^{*}1_{\mathcal{A}}1_{\mathcal{A}}^{*}+\alpha _{3} 1_{\mathcal{A}}1_{\mathcal{A}}^{*}1_{\mathcal{A}} +\alpha _{4} 1_{\mathcal{A}}1_{\mathcal{A}}1_{\mathcal{A}}^{*}\\
&+\alpha _{5} 1_{\mathcal{A}}1_{\mathcal{A}}1_{\mathcal{A}}+\alpha _{6} 1_{\mathcal{A}}1_{\mathcal{A}}^{*}1_{\mathcal{A}}^{*})=\alpha _{1} \Phi (1_{\mathcal{A}})\Phi (1_{\mathcal{A}})\Phi (1_{\mathcal{A}})+\alpha _{2} \Phi (1_{\mathcal{A}})^{*}\Phi (1_{\mathcal{A}})\Phi (1_{\mathcal{A}})^{*}\\
&+\alpha _{3} \Phi (1_{\mathcal{A}})\Phi (1_{\mathcal{A}})^{*}\Phi (1_{\mathcal{A}})+\alpha _{4} \Phi (1_{\mathcal{A}})\Phi (1_{\mathcal{A}})\Phi (1_{\mathcal{A}})^{*}+\alpha _{5} \Phi (1_{\mathcal{A}})\Phi (1_{\mathcal{A}})\Phi (1_{\mathcal{A}})\\
&+\alpha _{6}\Phi (1_{\mathcal{A}})\Phi (1_{\mathcal{A}})^{*}\Phi (1_{\mathcal{A}})^{*}=(\textstyle \sum _{k=1}^{6} \alpha _{k})\Phi (1_{\mathcal{A}}).
\end{align*}}
Hence, choose an element $c\in \mathcal{A},$ such that $\phi (c)=1_{\mathcal{B}}.$ Then 
{\allowdisplaybreaks\begin{align*}
&\Phi ((\textstyle \sum _{k=1}^{6} \alpha _{k})c)=\Phi (\alpha _{1} 1_{\mathcal{A}}1_{\mathcal{A}}c+\alpha _{2} 1_{\mathcal{A}}^{*}c1_{\mathcal{A}}^{*}+\alpha _{3} 1_{\mathcal{A}}1_{\mathcal{A}}^{*}c +\alpha _{4} c1_{\mathcal{A}}1_{\mathcal{A}}^{*}\\
&+\alpha _{5} 1_{\mathcal{A}}c1_{\mathcal{A}}+\alpha _{6} c1_{\mathcal{A}}^{*}1_{\mathcal{A}}^{*})=\alpha _{1} \Phi (1_{\mathcal{A}})\Phi (1_{\mathcal{A}})\Phi (c)+\alpha _{2} \Phi (1_{\mathcal{A}})^{*}\Phi (c)\Phi (1_{\mathcal{A}})^{*}\\
&+\alpha _{3} \Phi (1_{\mathcal{A}})\Phi (1_{\mathcal{A}})^{*}\Phi (c)+\alpha _{4} \Phi (c)\Phi (1_{\mathcal{A}})\Phi (1_{\mathcal{A}})^{*}+\alpha _{5} \Phi (1_{\mathcal{A}})\Phi (c)\Phi (1_{\mathcal{A}})\\
&+\alpha _{6}\Phi (c)\Phi (1_{\mathcal{A}})^{*}\Phi (1_{\mathcal{A}})^{*}=(\textstyle \sum _{k=1}^{6} \alpha _{k})\Phi (1_{\mathcal{A}})^{2}=\Phi ((\textstyle \sum _{k=1}^{6} \alpha _{k})1_{\mathcal{A}}).
\end{align*}}
This shows that $c=1_{\mathcal{A}}.$ As a consequence of this last result, for an arbitrary element $c\in \mathcal{A},$ we have
{\allowdisplaybreaks\begin{align*}
&\Phi ((\textstyle \sum _{k=1}^{6} \alpha _{k})c)=\Phi (\alpha _{1} 1_{\mathcal{A}}1_{\mathcal{A}}c+\alpha _{2} 1_{\mathcal{A}}^{*}c1_{\mathcal{A}}^{*}+\alpha _{3} 1_{\mathcal{A}}1_{\mathcal{A}}^{*}c +\alpha _{4} c1_{\mathcal{A}}1_{\mathcal{A}}^{*}\\
&+\alpha _{5} 1_{\mathcal{A}}c1_{\mathcal{A}}+\alpha _{6} c1_{\mathcal{A}}^{*}1_{\mathcal{A}}^{*})=\alpha _{1} \Phi (1_{\mathcal{A}})\Phi (1_{\mathcal{A}})\Phi (c)+\alpha _{2} \Phi (1_{\mathcal{A}})^{*}\Phi (c)\Phi (1_{\mathcal{A}})^{*}\\
&+\alpha _{3} \Phi (1_{\mathcal{A}})\Phi (1_{\mathcal{A}})^{*}\Phi (c)+\alpha _{4} \Phi (c)\Phi (1_{\mathcal{A}})\Phi (1_{\mathcal{A}})^{*}+\alpha _{5} \Phi (1_{\mathcal{A}})\Phi (c)\Phi (1_{\mathcal{A}})\\
&+\alpha _{6}\Phi (c)\Phi (1_{\mathcal{A}})^{*}\Phi (1_{\mathcal{A}})^{*}=(\textstyle \sum _{k=1}^{6} \alpha _{k})\Phi (c).
\end{align*}}
\end{proof}

\begin{claim}\label{s28} (i) $\Phi ((\alpha _{1}+\alpha _{3}+\alpha _{5})a)=(\alpha _{1}+\alpha _{3}+\alpha _{5})\Phi (a),$ for all element $a\in \mathcal{A}$ and (ii) $\Phi (b)^{*}=\Phi (b)^{*},$ for all element $b\in \mathcal{A}.$
\end{claim}
\begin{proof} It is clear that $\Phi ((\alpha _{1}+\alpha _{3}+\alpha _{5})a)=(\alpha _{1}+\alpha _{3}+\alpha _{5})\Phi (a),$ for all element $a\in \mathcal{A},$ because of hypothesis (ii), of the Theorem \ref{thm11}, and Claim \ref{s27}(ii). Thus, for an arbitrary element $b\in \mathcal{A}$ we have
{\allowdisplaybreaks\begin{align*}
&\Phi (\alpha _{1} 1_{\mathcal{A}}b1_{\mathcal{A}}+\alpha _{2} 1_{\mathcal{A}}^{*}1_{\mathcal{A}}b^{*}+\alpha _{3} b1_{\mathcal{A}}^{*}1_{\mathcal{A}} +\alpha _{4} 1_{\mathcal{A}}1_{\mathcal{A}}b^{*}+\alpha _{5} b1_{\mathcal{A}}1_{\mathcal{A}}+\alpha _{6} 1_{\mathcal{A}}b^{*}1_{\mathcal{A}}^{*})\\
&=\alpha _{1} \Phi (1_{\mathcal{A}})\Phi (b)\Phi (1_{\mathcal{A}})+\alpha _{2} \Phi (1_{\mathcal{A}})^{*}\Phi (1_{\mathcal{A}})\Phi (b)^{*}+\alpha _{3} \Phi (b)\Phi (1_{\mathcal{A}})^{*}\Phi (1_{\mathcal{A}})\\
&+\alpha _{4} \Phi (1_{\mathcal{A}})\Phi (1_{\mathcal{A}})\Phi (b)^{*}+\alpha _{5} \Phi (b)\Phi (1_{\mathcal{A}})\Phi (1_{\mathcal{A}})+\alpha _{6}\Phi (1_{\mathcal{A}})\Phi (b)^{*}\Phi (1_{\mathcal{A}})^{*}
\end{align*}}
that leads to the identity $(\alpha _{1}+\alpha _{3}+\alpha _{5})\Phi (b)+(\alpha _{2}+\alpha _{4}+\alpha _{6})\Phi (b^{*})=(\alpha _{1}+\alpha _{3}+\alpha _{5})\Phi (b)+(\alpha _{2}+\alpha _{4}+\alpha _{6})\Phi (b)^{*}.$ Consequently, we get $\Phi (b^{*})=\Phi (b)^{*}.$ 
\end{proof}

\begin{claim}\label{s29} $\Phi $ is a multiplicative map.
\end{claim}
\begin{proof} For arbitrary elements $b,c\in \mathcal{A},$ replace $a$ by $1_{\mathcal{A}}$ in the identity (\ref{fundident}). Then 
{\allowdisplaybreaks\begin{align*}
&\Phi (\alpha _{1} 1_{\mathcal{A}}bc+\alpha _{2} 1_{\mathcal{A}}^{*}cb^{*}+\alpha _{3} b1_{\mathcal{A}}^{*}c +\alpha _{4} c1_{\mathcal{A}}b^{*}+\alpha _{5} bc1_{\mathcal{A}}+\alpha _{6} cb^{*}1_{\mathcal{A}}^{*})\\
&=\alpha _{1}\Phi (1_{\mathcal{A}})\Phi (b)\Phi (c)+\alpha _{2}\Phi (1_{\mathcal{A}})^{*}\Phi (c)\Phi (b)^{*}+\alpha _{3}\Phi (b)\Phi (1_{\mathcal{A}})^{*}\Phi (c)\\
&+\alpha _{4}\Phi (c)\Phi (1_{\mathcal{A}})\Phi (b)^{*}+\alpha _{5}\Phi (b)\Phi (c)\Phi (1_{\mathcal{A}})+\alpha _{6}\Phi (c)\Phi (b)^{*}\Phi (1_{\mathcal{A}})^{*}
\end{align*}}
This results in the identity
{\allowdisplaybreaks\begin{align}\allowdisplaybreaks\label{id11}
&(\alpha _{1}+\alpha _{3}+\alpha _{5})\Phi (bc)+(\alpha _{2}+\alpha _{4}+\alpha _{6})\Phi (cb^{*})=(\alpha _{1}+\alpha _{3}+\alpha _{5})\Phi (b)\Phi (c)\nonumber\\
&+(\alpha _{2}+\alpha _{4}+\alpha _{6})\Phi (c)\Phi (b)^{*}.
\end{align}}
By applying involution to the identity (\ref{id11}), we get
{\allowdisplaybreaks\begin{align}\allowdisplaybreaks\label{id12}
&(\overline{\alpha _{1}+\alpha _{3}+\alpha _{5}})\Phi (c^{*}b^{*})+(\overline{\alpha _{2}+\alpha _{4}+\alpha _{6}})\Phi (bc^{*})=(\overline{\alpha _{1}+\alpha _{3}+\alpha _{5}})\Phi (c)^{*}\Phi (b)^{*}\nonumber\\
&+(\overline{\alpha _{2}+\alpha _{4}+\alpha _{6}})\Phi (b)\Phi (c)^{*}
\end{align}}
and, replacing in (\ref{id12}) $c^{*}$ by $c,$ we obtain 
{\allowdisplaybreaks\begin{align}\allowdisplaybreaks\label{id13}
&(\overline{\alpha _{2}+\alpha _{4}+\alpha _{6}})\Phi (bc)+(\overline{\alpha _{1}+\alpha _{3}+\alpha _{5}})\Phi (cb^{*})=(\overline{\alpha _{2}+\alpha _{4}+\alpha _{6}})\Phi (b)\Phi (c)\nonumber\\
&+(\overline{\alpha _{1}+\alpha _{3}+\alpha _{5}})\Phi (c)\Phi (b)^{*}.
\end{align}}
Multiplying (\ref{id11}) by the scalar $\overline{\alpha _{1}+\alpha _{3}+\alpha _{5}},$ (\ref{id13}) by the scalar $\alpha _{2}+\alpha _{4}+\alpha _{6}$ and subtracting the resulting identities, we arrive at $(|\alpha _{1}+\alpha _{3}+\alpha _{5}|^{2}-|\alpha _{2}+\alpha _{4}+\alpha _{6}|^{2})\Phi (bc)=(|\alpha _{1}+\alpha _{3}+\alpha _{5}|^{2}-|\alpha _{2}+\alpha _{4}+\alpha _{6}|^{2})\Phi (b)\Phi (c)$ which results in $\Phi (bc)=\Phi (b)\Phi (c).$ This shows that $\Phi $ is multiplicative.
\end{proof}
Therefore, by Claims \ref{s26}, \ref{s28}(ii) and \ref{s29} we conclude that $\Phi $ is a $\ast $-ring isomorphism.

The proof of the Theorem \ref{thm11} is complete.

From Theorem \ref{thm11} we can deduce the following result. However, we first present the necessary definitions and notations.

Let $\mathcal{A}$ and $\mathcal{B}$ be two complex $\ast $-algebras and $\eta $ a nonzero complex number. For $a,b\in \mathcal{A}$ (resp., $a,b\in \mathcal{B}$) denote by $a\vardiamondsuit _{\eta }b=ab+\eta ba^{*},$ the {\it Jordan $\eta $-$\ast $-product}. We say that a nonlinear map $\Phi:\mathcal{A}\rightarrow \mathcal{B}$ {\it preserves Jordan triple $\ast $-product $a\vardiamondsuit _{\eta }b\vardiamondsuit _{\nu }c$}, where $a\vardiamondsuit _{\eta }b\vardiamondsuit _{\nu }c=(a\vardiamondsuit _{\eta }b)\vardiamondsuit _{\nu }c,$ if $\Phi (a\vardiamondsuit _{\eta }b\vardiamondsuit _{\nu }c)=\Phi (a)\vardiamondsuit _{\eta }\Phi (b)\vardiamondsuit _{\nu }\Phi (c),$ for all elements $a,b,c\in \mathcal{A}.$ 

From the above definition, we can easily verify that nonlinear maps preserving Jordan triple $\ast $-products, as defined in \cite{Li1} and \cite{Zhao}, are nonlinear maps preserving Jordan triple $\ast $-products $a\vardiamondsuit _{1}b\vardiamondsuit _{1}c$ and $a\vardiamondsuit _{-1}b\vardiamondsuit _{-1}c,$ respectively, and nonlinear maps preserving Jordan triple $\ast $-product $a\vardiamondsuit _{\eta }b\vardiamondsuit _{\nu }c$ are nonlinear maps that preserve sum of triple products $1abc+0a^{*}cb^{*}+\eta ba^{*}c +\nu \overline{\eta} cab^{*}+0bca+\nu cb^{*}a^{*}.$

\begin{corollary} Let $\mathcal{A}$ and $\mathcal{B}$ be two unital complex $\ast $-algebras with $1_{\mathcal{A}}$ and $1_{\mathcal{B}}$ their multiplicative identities, respectively, and such that $\mathcal{A}$ is prime and has a nontrivial projection. Then every bijective nonlinear map $\Phi :\mathcal{A}\rightarrow \mathcal{B}$ preserving triple $\ast $-product $a\vardiamondsuit _{\eta }b\vardiamondsuit _{\nu }c,$ where $\eta ,\nu $ are nonzero complex numbers satisfying the conditions $\eta \neq -1$ and $|\nu |\neq 1,$ is additive. In addition, if (i) $\Phi (1_{\mathcal{A}})$ is a projection of $\mathcal{B}$ and (ii) $\Phi (\nu (\overline{\eta }+1)a)=\nu (\overline{\eta }+1)\Phi (a),$ for all element $a\in \mathcal{A},$ then $\Phi $ is a $\ast $-ring isomorphism. In particular, if $\Phi (1_{\mathcal{A}})$ is a projection of $\mathcal{B}$ and $\eta $ and $\nu $ are nonzero complex numbers such that $\nu (\overline{\eta }+1)$ is a rational number, then $\Phi $ is a $\ast $-ring isomorphism. 
\end{corollary}


\end{document}